\documentclass[11pt,a4paper,reqno]{amsart}
\usepackage{amssymb}

\usepackage{bm}
\numberwithin{equation}{section}

\usepackage{mathabx}
\usepackage{mathtools}
\usepackage{color}

\theoremstyle{plain}

\newtheorem{theorem}[equation]{Theorem}
\newtheorem{proposition}[equation]{Proposition}
\newtheorem{lemma}[equation]{Lemma}
\newtheorem{corollary}[equation]{Corollary}

\theoremstyle{definition}

\newtheorem{remark}[equation]{Remark}

\newtheorem{example}[equation]{Example}

\newcommand\R{\mathbb{R}}
\newcommand\Z{\mathbb{Z}}
\newcommand\N{\mathbb{N}}
\newcommand\Q{\mathbb{Q}}
\newcommand\eps{\varepsilon}
\newcommand\normsymb{\ell}  
\newcommand\norm[1]{\ell({#1})}   
\newcommand\normn[1]{\ell_n({#1})}   
\newcommand\free[1]{\mathbf{F}_{#1}}
\newcommand{\ab}{\mathrm{ab}}

\usepackage[breaklinks=true,hidelinks]{hyperref}

\begin{document}

\title{Homogeneous length functions on groups}

\author{D.H.J. Polymath}
\address{\url{http://michaelnielsen.org/polymath1/index.php}}

\begin{abstract}
A pseudo-length function defined on an arbitrary group $G = (G,\cdot,e,
(\,)^{-1})$ is a map $\normsymb: G \to [0,+\infty)$ obeying $\norm{e}=0$,
the symmetry property $\norm{x^{-1}} = \norm{x}$, and the triangle
inequality $\norm{xy} \leq \norm{x} + \norm{y}$ for all $x,y \in G$. We
consider pseudo-length functions which saturate the triangle inequality
whenever $x=y$, or equivalently those that are \emph{homogeneous} in the
sense that $\norm{x^n} = n\,\norm{x}$ for all $n\in\N$. We show that this
implies that $\norm{[x,y]}=0$ for all $x,y \in G$. This leads to a
classification of such pseudo-length functions as pullbacks from
embeddings into a Banach space.  We also obtain a quantitative version of
our main result which allows for defects in the triangle inequality or
the homogeneity property.
\end{abstract}



\maketitle

\section{Introduction}

Let $G = (G, \cdot, e, (\,)^{-1})$ be a group (written multiplicatively,
with identity element $e$).  A \emph{pseudo-length function} on $G$ is a
map $\normsymb : G \to [0,+\infty)$ that obeys the properties
\begin{itemize}
\item $\norm{e} = 0$,
\item $\norm{ x^{-1} } = \norm{x}$,
\item $\norm{ x y } \leq \norm{x} + \norm{ y }$
\end{itemize}
for all $x,y\in G$.  If in addition we have $\norm{x}>0$ for all $x \in G
\setminus \{e\}$, we say that $\normsymb$ is a \emph{length function}. By
setting $d(x,y) \coloneqq \norm{x^{-1}y}$, it is easy to see that
pseudo-length functions (resp. length functions) are in bijection with
left-invariant pseudometrics (resp. left-invariant metrics) on $G$.

From the above properties it is clear that one has the upper bound
\[
\norm{x^n} \leq |n| \, \norm{x}
\]
for all $x \in G$ and $n \in \Z$.  Let us say that a pseudo-length
function $\normsymb: G\to [0,+\infty)$ is \emph{homogeneous} if equality
is always attained here, in that one has
\begin{equation}\label{linear-growth}
\norm{x^n} = |n| \, \norm{x}
\end{equation}
for all $x \in G$ and any $n \in \Z$.  Using the axioms of a
pseudo-length function, it is not difficult to show that the homogeneity
condition \eqref{linear-growth} is equivalent to the triangle inequality
holding with equality whenever $x=y$ (i.e., that \eqref{linear-growth}
holds for $n=2$); see \cite[Lemma 1]{GK}.

If one has a real or
complex Banach space $\mathbb{B} = (\mathbb{B},\| \, \|)$, and $\phi: G
\to \mathbb{B}$ is any homomorphism from $G$ to $\mathbb{B}$ (viewing the
latter as a group in additive notation), then the function $\normsymb: G
\to [0,+\infty)$ defined by $\norm{x} \coloneqq \| \phi(x) \|$ is easily
verified to be a homogeneous pseudo-length function.  Furthermore, if
$\phi$ is injective, then $\normsymb$ is in fact a
homogeneous length function.  For instance, the function
$\norm{ (n, m) } \coloneqq |n + \sqrt{2} m|$
is a length function on $\Z^2$, where in this case $\mathbb{B} \coloneqq
\R$ and $\phi((n,m)) \coloneqq n  + \sqrt{2} m$.  On the other hand, one
can easily locate many length functions that are not homogeneous, for
instance by taking the square root of the length function just
constructed.

The main result of this paper is that such Banach space constructions are
in fact the \emph{only} way to generate homogeneous (pseudo-)length
functions.

\begin{theorem}[Classification of homogeneous length functions]\label{class}
Given a group $G$, let $\normsymb: G \to [0,+\infty)$ be a homogeneous
pseudo-length function.  Then there exist a real Banach space $\mathbb{B}
= (\mathbb{B},\| \, \|)$ and a group homomorphism $\phi : G \to
\mathbb{B}$ such that $\norm{x} = \| \phi(x) \|$ for all $x \in G$.
Furthermore, if $\normsymb$ is a length function, one can take $\phi$ to
be injective, i.e., an isometric embedding.
\end{theorem}

We derive Theorem~\ref{class} from a more quantitative result
bounding the pseudo-length of a commutator
\begin{equation}\label{comdef}
[x,y] \coloneqq xyx^{-1}y^{-1};
\end{equation}
see Proposition~\ref{main} below.  Our arguments are elementary, relying
on directly applying the axioms of a homogeneous length function to
various carefully chosen words in $x$ and $y$, and repeatedly taking an
asymptotic limit $n \to \infty$ to dispose of error terms that arise in
the estimates obtained in this fashion.

An additional advantage of quantifying Theorem~\ref{class}
in Proposition~\ref{main} is that one can derive from the latter
proposition a ``quasified'' version of Theorem~\ref{class}. See
Theorem~\ref{quasi} below.\footnote{A different variant of
Theorem~\ref{class} involves replacing homogeneity by the assumption that
$\normsymb$ is a pseudo-length function on $G$ whose homogenization is
positive:
\[
\normsymb_{\rm hom}(g) \coloneqq \lim_{n \to \infty} \frac{\norm{g^n}}{n}
> 0, \qquad \forall g \neq e.
\]
(This was studied in \cite[Theorem 2.10(III)]{Niem} in the special case
of abelian $(G,\normsymb)$.) In this case we work with $(G,
\normsymb_{\rm hom})$ instead of $\normsymb$, to conclude that $G$ maps
into a Banach space.}

Finally, as one quick corollary of Theorem \ref{class}, we obtain the
following characterization of the groups that admit homogeneous length
functions.

\begin{corollary}\label{abel}
A group admits a homogeneous length function if and only if it is
abelian and torsion-free.
\end{corollary}

\subsection{Examples and approaches}

We now make some remarks to indicate the nontriviality of
Theorem~\ref{class}. Corollary \ref{abel} implies that there are no
non-abelian groups with homogeneous length functions.
Whether or not such a striking geometric rigidity phenomenon holds was
previously unknown to experts. Moreover, the corollary fails
to hold if one or more of the precise conditions in the theorem are
weakened. For instance, such length functions indeed exist (i)~on
non-abelian monoids, and (ii)~on balls of finite radius in free groups.
We explain these two cases further in Section \ref{remarks-sec}.

Given these cases, one could \textit{a priori} ask if every non-abelian
group admits a homogeneous length function. This is not hard to disprove;
here are two examples.

\begin{example}[Nilpotent groups]
If $G$ is nilpotent of nilpotency class two (e.g., the Heisenberg group),
then $[x,y]^{n^2} = [x^n,y^n]$ for all $x,y \in G$ and integers $n \geq
0$ since the map $(g,h) \mapsto [g,h]$ is now a bihomomorphism
$G \times G \to [G,G] \subset Z(G)$.
If $[x,y]$ is non-trivial, then any homogeneous length function on $G$
would assign a linearly growing quantity to the right-hand side and a
quadratically growing quantity to the left-hand side, which is absurd;
thus such groups cannot admit homogeneous length functions.  The claim
then also follows for nilpotent groups of higher nilpotency class, since
they contain subgroups of nilpotency class two.\footnote{One can also
show by relatively simple means that solvable non-abelian groups cannot
admit homogeneous length functions either; see the discussion on
lamplighter groups in the comments to
\href{https://terrytao.wordpress.com/2017/12/16/bi-invariant-metrics-of-linear-growth-on-the-free-group/}{\tt
terrytao.wordpress.com/2017/12/16/}.}
\end{example}

\begin{example}[Connected Lie groups]
As we explain in Remark~\ref{Rinv}, a homogeneous length function
$\normsymb$ induces a \textit{bi-invariant} metric on $G$.
Now if $(G, \normsymb)$ is furthermore a connected Lie group, then by
\cite[Lemma 7.5]{Milnor}, $G \cong K \times \R^n$ for some compact Lie
group $K$ and integer $n \geq 0$. By \eqref{linear-growth}, $K$ cannot
have torsion elements, hence must be trivial. But then $G$ is abelian.
\end{example}

Prior to Corollary \ref{abel}, the above examples left open the question
of whether \textit{any} non-abelian group admits a homogeneous length
function. One may as well consider groups generated by two non-commuting
elements. As a prototypical example, let $\free{2}$ be the free group on
two generators $a,b$.  The word length function on $\free{2}$ is a length
function, but it is not homogeneous, since for instance the word length
of $(bab^{-1})^n = ba^n b^{-1}$ is $n+2$, which is not a linear function
of $n$.  It is however the case that the word length of $x^n$ has linear
\emph{growth} in $n$ for any non-trivial $x$. Similarly for the
Levenshtein distance (edit distance) on $\free{2}$.

Our initial attempts to construct homogeneous length functions on
$\free{2}$ all failed. Of course, this failure is explained by our main
result. However, many of these methods apply under minor weakening of the
hypotheses, such as working with monoids rather than groups, or weakening
homogeneity. Results in these cases are discussed further in Section
\ref{remarks-sec}.

\subsection{Further motivations}

We next mention some motivations from functional analysis and
probability, or more precisely the study of Banach space embeddings.
If $G$ is an additive subgroup of a Banach space $\mathbb{B}$, then
clearly the norm on $\mathbb{B}$ restricts to a homogeneous length
function on $G$. In~\cite{CSC,GK} one can find several equivalent
conditions for a given length function on a given group to arise in this
way (studied in the broader context of additive mappings and separation
theorems in functional analysis); see also \cite[Theorem 2.10(II)]{Niem}
for an alternative proof. These conditions are summarized in \cite{KR2}.
For instance, given a group $G$ with a length function $\ell$, there
exists an isometric embedding from $G$ to a Banach space $\mathbb{B}$
with $\normsymb$ induced from the metric on $\mathbb{B}$, if and only if
$G$ is amenable and $\norm{x^2} = 2 \norm{x}$ for all $x$.

In view of such equivalences, it is natural to try and characterize the
groups possessing a homogeneous length function. This question is
answered by Corollary~\ref{abel}, which shows these are precisely the
abelian torsion-free groups.

Groups and semigroups with translation-invariant metrics also naturally
arise in probability theory, with the most important `normed' (i.e.,
homogeneous) examples being Banach spaces~\cite{LT}. Notice however that
in certain fundamental stochastic settings, formulating and proving
results does not require the full Banach space structure. In this vein, a
general variant of the Hoffmann-J{\o}rgensen inequality was shown in
\cite{KR1} in arbitrary metric semigroups -- including Banach spaces as
well as (non-abelian) compact Lie groups. Similarly in \cite{KR2}, the
authors transferred the (sharp) Khinchin--Kahane inequality from Banach
spaces to abelian groups $G$ equipped with a homogeneous length function.
To explore extensions of these results to the non-abelian setting (e.g.,
Lie groups with left-invariant metrics), we need to first understand if
such objects exist. As explained above, this question was not answered in
the literature; but it is now settled by our main result.

Finally, there may also be a relation to the Ribe program \cite{Naor},
which aims to reformulate aspects of Banach space theory in purely metric
terms. Indeed, from Corollary~\ref{abel} we see that a metric space $X$
is isometric to an additive subgroup of a Banach space if and only if
there is a group structure on $X$ which makes the metric left-invariant
and the length function $\norm{ x } \coloneqq d(1,x)$ homogeneous.

\section*{Acknowledgements}

This project is an online collaboration that originated from a blog post
at \url{https://terrytao.wordpress.com/2017/12/16}, following the model
of the ``Polymath'' projects \cite{gn}.  A full list of participants and
their grant acknowledgments may be found at
{\small
\url{http://michaelnielsen.org/polymath1/index.php?title=linear\_norm\_grant\_acknowledgments}}.
We also thank Michal Doucha for useful references and comments, in
particular in bringing the paper \cite{Niem} to our attention, and the
anonymous referee for helpful suggestions.

\section{Key proposition}

The key proposition used to prove Theorem~\ref{class} is the following
estimate, which can treat a somewhat more general class of functions than
homogeneous pseudo-length functions, in which the symmetry hypothesis is
dropped and one allows for an error in { the homogeneity
property}, which is now { also} only claimed for $n=2$.

\begin{proposition}\label{main}
Let $G = (G,\cdot)$ be a group, let { $c \in \R$}, and let
$\normsymb: G \to \R$ be a function obeying the following axioms:
\begin{itemize}
\item[(i)]  For any $x,y \in G$, one has
\begin{equation}\label{ng}
 \norm{xy} \leq \norm{x}+\norm{y}.
\end{equation}
\item[(ii)]  For any $x \in G$, one has
\begin{equation}\label{double}
 \norm{x^2} \geq 2\norm{x} - { c}.
\end{equation}
\end{itemize}
Then for any $x,y \in G$, one has
\begin{equation}\label{xyc}
 \norm{ [x,y] } \leq { 5 c},
\end{equation}
where the commutator $[x,y]$ was defined in \eqref{comdef}.
\end{proposition}

Notably, we neither assume symmetry $\norm{ x^{-1} } = \norm{ x }$, not
even up to a constant, nor $\norm{ e } = 0$ { (although $0
\leq \norm{ e } \leq c$ follows from the axioms)}; we also allow
$\normsymb$ to take on negative values. The reader may however wish to
restrict attention to { homogeneous} length functions, and
set { $c=0$ and $\normsymb \geq 0$} for a first reading of
the arguments below.  The { factor of $5$ is} probably not
optimal here, but the crucial feature of the bound \eqref{xyc} for our
main application is that the right-hand side vanishes when {
$c=0$} (the right-hand side is also independent of $x$ and $y$, which we
use in other applications).

{ We define a \emph{semi-length function} to be a function
$\normsymb: G \to \R$ such that for all $x, y\in G$, $\norm{xy} \leq
\norm{x} + \norm{y}$, i.e.  $\normsymb$ satisfies ~\eqref{ng}. Every
pseudo-length function is a semi-length function. A semi-length function
that satisfies \eqref{double} for some $c\in\R$ is called
\emph{quasi-homogeneous}.

\begin{remark}\label{quasilength}
Suppose $\normsymb: G \to \R$ and there is a constant $k$ such that
$\norm{xy} \leq \norm{x} + \norm{y} + k$ for all $x, y\in G$. Then the
function $\normsymb'(x) := \norm{x} + k$ is a semi-length {
function}. Further,
$\normsymb'$ satisfies~\eqref{double} with $c$ replaced by $c' := k+c$,
{ whenever $\normsymb$ satisfies~\eqref{double} on the nose.}
Thus Proposition~\ref{main} continues to hold if \eqref{ng} is replaced
by the condition $\norm{xy} \leq \norm{x} + \norm{y} + k$ for all $x,
y\in G$, with the bound in the conclusion \eqref{xyc} becoming $5c + 4k$.
\end{remark}
}

We now turn to the proof. For the remainder of this section, let $G$,
{ $c$}, and $\normsymb$ satisfy the hypotheses of the
proposition.  Our task is to establish the bound \eqref{xyc}.  We shall
now use \eqref{ng} and \eqref{double} repeatedly to establish a number of
further inequalities relating the { semi-}lengths $\norm{x}$ of various
elements $x$ of $G$, culminating in \eqref{xyc}.
Many of our inequalities will involve terms that depend on an auxiliary
parameter $n$, but we will be able to eliminate several of them by the
device of passing to the limit $n \to \infty$.  It is because of this
device that we are able to obtain a bound \eqref{xyc} whose right-hand
side is completely uniform in $x$ and $y$.

From \eqref{ng} and induction we have the upper homogeneity
bound
\begin{equation}\label{ng-iter}
 \norm{x^n} \leq { n \norm{x}}
\end{equation}
for any natural number $n \geq 1$.  Similarly, from \eqref{double} and
induction one has the lower homogeneity bound
\[
\norm{x^n} \geq { n \norm{x} - \log_2(n)\, c \geq n \norm{x} - n c}
\]
whenever $n$ is a power of two.  It is convenient to rearrange this
latter inequality as
\begin{equation}\label{rearrange}
\norm{x} \leq \frac{\norm{x^n}}{n}  + { c}.
\end{equation}
This inequality, particularly in the asymptotic limit $n \to \infty$,
will be the principal means by which the hypothesis \eqref{double} is
employed.

We remark that by further use of \eqref{ng-iter} one can also obtain a
similar estimate to \eqref{rearrange} for natural numbers $n$ that are
not powers of two, but the powers of two will suffice for the
arguments that follow.

\begin{lemma}[Approximate conjugation invariance]\label{aci}
For any $x,y \in G$, one has
\[
\norm{yxy^{-1} } \leq { \norm{ x } + c}.
\]
\end{lemma}

\begin{remark}\label{Rinv}
Setting { $c=0$}, we conclude that any homogeneous
pseudo-length function is conjugation invariant, and thus determines a
bi-invariant metric on $G$. It should not be surprising that this
observation is used in the proof of Theorem~\ref{class}, since it is a
simple consequence of that theorem.
\end{remark}

\begin{proof}[Proof of Lemma~\ref{aci}]
From \eqref{rearrange} with $x$ replaced by $yxy^{-1}$, one has
\[
\norm{yxy^{-1}} \leq \frac{\norm{y x^n y^{-1}} }{n} + { c}
\]
whenever $n$ is a power of two.  On the other hand, from \eqref{ng-iter}
and \eqref{ng} one has
\[
\norm{ y x^n y^{-1} } \leq \norm{y} + n \norm{x} + \norm{y^{-1}}
\]
and thus
\[
\norm{yxy^{-1}} \leq \norm{x} + { c} + \frac{\norm{y} +
\norm{y^{-1}} { - c}}{n}.
\]
Sending $n \to \infty$, we obtain the claim.
\end{proof}

\begin{lemma}[Splitting lemma]\label{split}
Let $x,y,z,w \in G$ be such that $x$ is conjugate to both $wy$ and
$zw^{-1}$.  Then one has
\begin{equation}\label{ineq}
\norm{x} \leq \frac{\norm{y} + \norm{ z } }{2}  + \frac{3}{2}
{ c}.
\end{equation}
\end{lemma}

\begin{proof}
If we write $x = swys^{-1} = t zw^{-1} t^{-1}$ for some $s,t \in G$, then
from \eqref{rearrange} we have
\begin{align*}
 \norm{x} &\leq \frac{\norm{ x^n x^n } }{2n} + { c} \\
&= \frac{\norm{ s(wy)^n s^{-1}t (zw^{-1})^n t^{-1} } }{2n} +
{ c}
\end{align*}
whenever $n$ is a power of two.  From Lemma~\ref{aci} and \eqref{ng} one
has
\begin{align*}
\norm{(wy)^{k+1} s^{-1}t (zw^{-1})^{k+1} } &= \norm{ w y (wy)^k s^{-1} t
(zw^{-1})^k z w^{-1} } \\
&\leq \norm{ y (wy)^k s^{-1} t (zw^{-1})^k z } + { c} \\
&\leq \norm{ (wy)^k s^{-1} t (zw^{-1})^k } + \norm{y} + \norm{z} +
{ c}
\end{align*}
for any $k \geq 0$, and hence by induction
\[
\norm{(wy)^n s^{-1}t (zw^{-1})^n } \leq \norm{s^{-1} t} + n ( \norm{y} +
\norm{z} + { c} ).
\]
Inserting this into the previous bound for $\norm{x}$ via two
applications of \eqref{ng}, we conclude that
\[
\norm{x} \leq \frac{\norm{y} + \norm{z}+ { c} }{2} +
\frac{\norm{s} + \norm{s^{-1}t} + \norm{t^{-1}} }{2n} + { c};
\]
sending $n \to \infty$, we obtain the claim.
\end{proof}

\begin{corollary}
If $x,y \in G$, let $f = f_{x,y}: \Z^2 \to \R$ denote the function
\[
f(m,k) \coloneqq \norm{ x^m [x,y]^k }.
\]
Then for any $m,k \in \Z$, we have
\begin{equation}\label{fmk}
 f(m,k) \leq \frac{f(m-1,k) + f(m+1,k-1)}{2} + { 2c}.
\end{equation}
\end{corollary}

\begin{proof}
Observe that $x^m [x,y]^k$ is conjugate to both $x (x^{m-1} [x,y]^k)$ and
to $(y^{-1} x^m [x,y]^{k-1} xy) x^{-1}$, hence by \eqref{ineq} one has
\[
\norm{ x^m [x,y]^k } \leq \frac{\norm{ x^{m-1} [x,y]^k } + \norm{ y^{-1}
x^{m} [x,y]^{k-1} xy }}{2} + { \frac{3}{2} c}.
\]
Since $y^{-1} x^{m} [x,y]^{k-1} xy $ is conjugate to $x^{m+1}
[x,y]^{k-1}$, the claim now follows from Lemma~\ref{aci}.
\end{proof}

We now prove Proposition~\ref{main}. Let $x,y \in G$.  We can
write the inequality \eqref{fmk} in probabilistic form as
\[
f(m,k) \leq {\mathbf E} f\left( \left(m,k-\frac{1}{2}\right) + Y
\left(1,-\frac{1}{2}\right) \right) + { 2 c}
\]
where $Y = \pm 1$ is a Bernoulli random variable that equals $1$ or $-1$
with equal probability.  The key point here is the drift of $\left(0,
-\frac{1}{2} \right)$ in the right-hand side.
Iterating this inequality, we see that
\[
f( 0, n) \leq {\mathbf E} f\left( (Y_1 + \dots + Y_{2n}) \left(1,
-\frac{1}{2}\right) \right) + { 4 c n},
\]
where $n \geq 0$ and $Y_1,\dots,Y_{2n}$ are independent copies of $Y$ (so
in particular $Y_1+\dots+Y_{2n}$ is an even integer).

From \eqref{ng} and \eqref{ng-iter} one has the inequality
\begin{eqnarray*}
f(m,k) & \leq &
|m| \left(\max( \norm{x}, \norm{x^{-1}}) \right)\\
& & + |k| \left(\max(
\norm{[x,y]}, \norm{[x,y]^{-1}}) \right) + \norm{e}
\end{eqnarray*}
for all integers $m,k$, where the final term $\norm{e}$ is used when
$m=k=0$, { but can also be added in the remaining cases since
it is non-negative}.  We conclude that
\[
f\left( (Y_1 + \dots + Y_{2n}) \left(1, -\frac{1}{2}\right) \right) \leq
A |Y_1 + \dots + Y_{2n}| + \norm{e}
\]
where $A$ is a quantity independent of $n$; more explicitly, one can take
\[
A \coloneqq \max\left( \norm{x}, \norm{x^{-1}}\right) + \frac{1}{2}
\max\left(\norm{[x,y]}, \norm{[x,y]^{-1}}\right).
\]
Taking expectations, since the random variable $Y_1+\dots+Y_{2n}$ has
mean zero and variance $2n$, we see from the Cauchy--Schwarz inequality
or Jensen's inequality that
\[
{\mathbf E}|Y_1 + \dots + Y_{2n}|  \leq \left( {\mathbf E}|Y_1 + \dots +
Y_{2n}|^2 \right)^{1/2} = \sqrt{2n}
\]
and hence
\[
f(0,n) \leq A \sqrt{2n} + \norm{e} + { 4 c n}.
\]
But from \eqref{rearrange}, if $n$ is a power of $2$ then we have
\[
\norm{[x,y]} \leq \frac{f(0,n)}{n}  + { c}.
\]
Combining these two bounds and sending $n \to \infty$, we obtain
Proposition~\ref{main}.



\begin{remark}
{
One can deduce a `local' version of Proposition \ref{main} as follows:
notice that the constant $c$ can be described in terms of $\normsymb$
from \eqref{double}, to yield
\begin{equation}\label{9bound}
\norm{ [x,y] } \leq 5 \sup_{z \in G} \left( 2 \norm{z} - \norm{z^2} \right)
\end{equation}
for any group $G$ and function $\normsymb : G \to \R$ for which this supremum exists, and any $x,y \in G$.
(Both sides are zero when $G$ is a Banach space and $\normsymb$ is the
norm, so equality is obtained in that case.)
It is also enough to consider the supremum over the subgroup of $G$
generated by $x$ and $y$ without loss of generality, which may lead to a
better bound on $\norm{ [x,y] }$ than taking the supremum over all of $G$.
Notice also that the constant $c$ must be non-negative, from \eqref{double} and \eqref{ng} with $x=y=e$:
\begin{equation}\label{Enonneg2}
c \geq 2 \norm{e} - \norm{e^2} = \norm{e} \geq \norm{e^2} - \norm{e} = 0.
\end{equation}
In fact, this reasoning and our results imply that the only way to get
$c=0$ on the right-hand side of \eqref{xyc} is when $\normsymb$
arises from pulling back the norm of a Banach space $\mathbb{B}$ along a
group homomorphism $G\to\mathbb{B}$, or equivalently along a group
homomorphism from the torsion-free abelianization of $G$ to $\mathbb{B}$.
}
\end{remark}

\section{Completing the proof of Theorem~\ref{class}}

With Proposition~\ref{main} in hand, it is not difficult to conclude the
proof of Theorem~\ref{class}.  Suppose that $G$ is a group with a
homogeneous { semi-}length function $\normsymb: G \to [0,+\infty)$.
Applying Proposition~\ref{main} with { $c=0$}, we conclude that
$\norm{[x,y]}=0$ for all $x,y \in G$, thus by the triangle inequality
$\normsymb$ vanishes on the commutator subgroup $[G,G]$, and therefore
factors through the abelianization $G_{\mathrm{ab}} \coloneqq G/[G,G]$ of
$G$.  Observe that this already establishes part of one implication of
Corollary~\ref{abel}.  Factoring out by $[G,G]$ like this, we may now
assume without loss of generality that $G$ is abelian.  To
reflect this, we now use additive notation for $G$, thus for instance
$\norm{nx} = |n| \norm{x}$ for each $x \in G$ and $n \in \Z$, and one can
also view $G$ as a module over the integers $\Z$.

At this point we repeat the arguments in \cite[Theorem B]{KR2}, which
treated the case when $G$ was separable, though it turns out that this
separability hypothesis is unnecessary.

If $x$ is a torsion element of $G$, i.e.~$nx=0$ for some $n$, then the
homogeneity condition forces $\norm{x}=0$.  Thus $\normsymb$ vanishes on
the torsion subgroup of $G$; factoring out by this subgroup, we may thus
assume without loss of generality that $G$ is not only abelian, but is
also torsion-free.

We can view $G$ as a subgroup of the $\Q$-vector space $G \otimes_\Z \Q$,
the elements of which can be formally expressed as $\frac{1}{n} x$ for
natural numbers $n$ and elements $x \in G$ (with two such expressions
$\frac{1}{n} x, \frac{1}{m} y$ identified if and only if $mx = ny$, and
the $\Q$-vector space operations defined in the obvious fashion); the
fact that this is well defined as a $\Q$-vector space follows from the
hypotheses that $G$ is abelian and torsion-free.  We can then define
the map $\| \, \|_\Q: G \otimes_\Z \Q \to [0,+\infty)$ by setting
\[
\left\|\frac{1}{n} x\right\|_\Q \coloneqq \frac{1}{n} \norm{x}
\]
for any $x \in G$ and natural number $n$; the linear growth condition
ensures that $\|\,\|_\Q$ is well-defined.  It is not difficult to verify
that $\| \, \|_\Q$ is indeed a seminorm over the $\Q$-vector space $G
\otimes_\Z \Q$.

The norm $\| \, \|_\Q$ on $G \otimes_\Z \Q$ gives a metric $d(x, y) = \|
x-y \|_\Q$. Consider the metric completion $\mathbb{B}$ of $G \otimes_\Z
\Q$ with this metric. It is easy to see that the $\Q$-vector space
structure on $G \otimes_\Z \Q$ extends to an $\R$-vector space structure
on $\mathbb{B}$, and the norm $\|\,\|_\Q$  on $G \otimes_\Z \Q$ extends
to a norm $\|\,\|_\R$  on $\mathbb{B}$. As $\mathbb{B}$ is complete by
construction, it is a Banach space. The inclusion of $G$ in $G \otimes_\Z
\Q$ gives a homomorphism $\phi: G \to \mathbb{B}$ as required.

This concludes the proof of Theorem~\ref{class}.  Since the homomorphism
$\phi: G \to \mathbb{B}$ can only be injective for abelian torsion-free
$G$, we obtain the ``only if'' portion of Corollary~\ref{abel}.
Conversely, if a group $G$ is abelian and torsion-free, by the above
constructions it embeds into a real vector space $\mathbb{B} \coloneqq G
\otimes_\Z \R$; now by Zorn's lemma $\mathbb{B}$ has a norm (e.g.,
consider the $\ell^1$ norm with respect to a Hamel basis of
$\mathbb{B}$), which restricts to the desired homogeneous length function
on $G$.
We remark that $G \otimes_\Z \R$ is the construction of the smallest,
`enveloping' vector space containing a copy of the abelian, torsion-free
group $G$.

\begin{remark}
The above arguments also show that homogeneous pseudo-length functions on
$G$ are in bijection with seminorms on the real vector space
$G_{\mathrm{ab},0} \otimes_\Z \R$, where $G_{\mathrm{ab},0}$ denotes the
torsion-free abelianization of $G$.
\end{remark}

\section{Further remarks and results}\label{remarks-sec}

If we weaken any of several conditions in Corollary \ref{abel}, then
examples of non-abelian structures with generalized length functions do,
in fact, often exist.
However, the generality of Proposition 2.1 allows us to obtain
non-trivial information in some of these cases. Here we mention several
such cases and discuss other related problems.

\subsection{Monoids and embeddings}

Our first weakening is to replace `groups' by the more primitive
structures `monoids' or `semigroups'. In this case, Robert Young (private
communication) described to us non-abelian monoids with homogeneous,
bi-invariant length functions:
consider the free monoid $\mathrm{FMon}(X)$ on any alphabet $X$ of size
at least $2$, with the edit distance $d(v,w)$ between strings $v,w \in
\mathrm{FMon}(X)$ being the least number of single generator insertions
and deletions to get from $v$ to $w$. The triangle inequality and
positivity are easily verified, while homogeneity of the corresponding
length function $\norm{x} \coloneqq d(e,x)$ is trivial. Moreover, the
metric $d(\cdot,\cdot)$ turns out to be bi-invariant:
\[
d(gxh,gyh) = d(x,y) \ \text{ for all }\ g,h,x,y \in \mathrm{FMon}(X).
\]
This specializes to left- and right-invariance upon taking $g \in X$ and
$h = e$, or  $h \in X$ and $g = e$, respectively.

Note moreover that $\mathrm{FMon}(X)$ embeds into the free group
$\mathrm{FGp}(X)$ generated by $X$ and $X^{-1}$, where $X^{-1}$ is the
collection of symbols defined to be inverses of elements of $X$. In
particular, $\mathrm{FMon}(X)$ is cancellative. While this trivially
addresses the embeddability issue, notice that a more refined version of
embeddability fails. Namely, by our main theorem, $\mathrm{FMon}(X)$ does
not embed into any group in the category $\mathcal{C}_{\rm{bi-inv,\
hom}}$ with cancellative semigroups with homogeneous bi-invariant metrics
as objects and isometric semigroup maps as morphisms. Thus, one may
reasonably ask what is a sufficiently small category in which the
embeddability works. The following proposition shows that we just need to
drop homogeneity.

\begin{proposition}\label{Pisometry}
Let $\mathcal{C}_{\rm{bi-inv}}$ denote the category whose objects are
cancellative semigroups with bi-invariant metrics, and morphisms are
isometric semigroup maps. Then $\mathrm{FMon}(X)$ embeds isometrically
into $\mathrm{FGp}(X)$ in $\mathcal{C}_{\rm{bi-inv}}$.
\end{proposition}

\begin{proof}
From above, $\mathrm{FMon}(X)$ is an object of
$\mathcal{C}_{\rm{bi-inv}}$; denote the metric by $d_{FM}$. One can check
that $d_{FM}(w,w')$ equals the difference between $\normsymb(w) +
\normsymb(w')$ and twice the length of the longest common (possibly
non-contiguous) substring in $w,w'$; here, $\normsymb$ denotes the length
of a word in the alphabet $X$.

We next claim $\mathrm{FGp}(X)$ is also an object of
$\mathcal{C}_{\rm{bi-inv}}$. Namely, for a word $w=x_1x_2\cdots x_m$ in
the free group, we consider \emph{non-crossing matchings} in $w$, i.e.,
sets $M$ of pairs of letters in $\{1,2,\dots m\}$ such that the following
hold.
\begin{itemize}
\item If $(i, j)\in M$, then $i< j$ and  $x_j=x_i^{-1}$.
\item If $(i, j), (k, l) \in M$, then either $(i, j) = (k, l)$ or
$i,j,k,l$ are distinct.
\item If $i<k<j<l$ and $(i, j)\in M$, then $(k, l)\notin M$.
\end{itemize}

Given a matching $M$ as above, consider the set $U=U(M)$ of indices $k$,
$1\leq k\leq m$ which are not part of a pair in $M$. Define the
\emph{deficiency} of the matching $M$ as the cardinality of the set
$U(M)$, and define the length $\normsymb_{wc}(w)$ of the word $w$ as the
infimum of the deficiency over all non-crossing matchings in $w$ (the
subscript in $\normsymb_{wc}$ stands for Watson--Crick). This length was
previously studied in \cite{gadgil}, including checking that it is
well-defined on all of $\mathrm{FGp}(X)$; moreover, $\normsymb_{wc}(w)$
equals the smallest number of conjugates of elements in $X \sqcup X^{-1}$
whose product is $w$. Now define $d_{FG}(w,w') \coloneqq
\normsymb_{wc}(w^{-1} w')$. It is easy to see that $\normsymb_{wc}$ is a
conjugacy invariant length function.

We claim that $d_{FG} \equiv d_{FM}$ on $\mathrm{FMon}(X)$, which proves
the result. It is easy to show that if two words in $\mathrm{FMon}(X)$
differ by a single insertion or deletion, then their distance in
$\mathrm{FGp}(X)$ is at most one, hence exactly one. In the other
direction, we claim that a non-crossing matching on $w^{-1}w'$, with $w$
and $w'$ containing only positive generators (in $X$), is just a
`rainbow', i.e.~nested arches with one end in $w^{-1}$ and the other in
$w'$. But then $d_{FG}(w,w')$ equals $\normsymb(w) + \normsymb(w')$ minus
twice the length of a common substring, which is maximal by the
minimality of the deficiency. Hence $d_{FG}(w,w') = d_{FM}(w,w')$,
completing the proof.
\end{proof}

{ Note that given weights $\normsymb(a)$ and $\normsymb(b)$,
there is a natural weighted version $\normsymb_{wc;a,b}$ where the
letters of $U$ as above are taken with these weights (symmetrically under
inversion). This corresponds to the weighted edit distance, with
different costs for editing different letters.}


\subsection{ Quasimorphisms and commutator lengths}

{ We now investigate potential applications of Proposition \ref{main} with $c > 0$.}
A \emph{quasi-morphism} on a group $G$ is a map $f : G \to \R$
whose \emph{defect} is bounded,
\[
	D(f) \coloneqq \sup_{x,y\in G} | f(xy) - f(x) - f(y) | < +\infty.
\]
{ Every quasi-morphism induces a pseudo-length function {(in particular semi-length function)}
by setting
\begin{equation}\label{quasimortonorm}
	\norm{ x } \coloneqq | f(x) | { + D(f),}
\end{equation}
where we can take $c = 2 D(f)$ as a bound on the homogeneity defect. In this case, Proposition~\ref{main} makes a rather
trivial statement: a homogeneous quasi-morphism is bounded on
commutators,
\[
	| f([x,y]) | \leq { 10} D(f).
\]
In fact, as observed in \cite[Lemme~1.1]{Br}, for homogeneous quasi-morphisms one can improve the
constant from $10$ to $3$, and a quasi-morphism can always be homogenized by replacing it by $\lim_{n \to \infty} f( x^n) / n$~\cite[p.~135]{Br}, which differs from the original $f$ by at most $D(f)$.}

Nevertheless, quasi-morphisms can be utilized to construct interesting
{ pseudo-length functions}, for example satisfying homogeneity on
specific commutators. The following quasi-morphism is due to
Brooks~\cite[Section~2]{Fu}.
For a given word $w$ in the free group $\free{2}$, written in reduced
form, let $f_w : \free{2} \to \R$ be the function which assigns to every
other $g\in \free{2}$, also written in reduced form, the maximum number
of times such that $w$ occurs in $g$ without overlaps, minus the
analogous number of times that $w^{-1}$ can maximally occur in $g$. Since
$f_w(w^n) = n f_w(w)$, { using~\eqref{quasimortonorm} results in a
pseudo-length function that grows linearly on the powers of $w$}. For example with $w$ being the
commutator of the generators of $\free{2}$, we see that although the
{ pseudo-length} function must be bounded on commutators by
{ Proposition~\ref{main}}, it can nevertheless grow linearly
on the powers of a fixed commutator.


Thus, there exist { examples of quasi-homogeneous semi-length functions on free groups} that are not induced by
norms. Nevertheless, we { will now} see that for a large class of groups, including
amenable groups and $G = SL(n, \Z)$ for $n\geq 3$, { even all quasi-homogeneous { semi-}length} functions are induced by norms on Banach spaces.
Further, the bound from { Proposition~\ref{main}} even in
the case of free groups is sharper than that obtained without using
homogeneity.

Recall that the \emph{commutator length} $\mathrm{cl}(g)$ of a word in
$[G, G]$ is the length $k$ of the shortest expression $g = [a_1,
b_1]\cdot[a_2, b_2]\cdots[a_k, b_k]$ of $g$ as a product of commutators.
The \emph{stable commutator length} is defined as $\lim_{n\to\infty}
\mathrm{cl}(g^n)/n$, where the limit exists by sub-additivity of the
function $n\mapsto \mathrm{cl}(g^n)$.

Then { Proposition~\ref{main}}, together with {$\norm{e} \leq c$ and
\eqref{rearrange} for $n$ a power of two,
\[
\norm{x} \leq \frac{\norm{x^n}}{n}  + \log_2(n)\, c,
\]}
easily imply the following estimates:

\begin{proposition}\label{commutatorbound}
{
Let $\normsymb$ and $c$ be as in
Proposition~\ref{main}. Then for
$x\in [G, G]$, $\norm{x}\leq (5\, \mathrm{cl}(x) + 1)c$
and $\norm{x}\leq (5\, \mathrm{scl}(x) + 1)c$.
}
\end{proposition}

We say two { { semi-}length} functions $\normsymb_1,\normsymb_2: G
\to\R$ are \emph{equivalent} if $|\normsymb_1(x) - \normsymb_2(x)|$ is bounded in $x\in G$.

For a group $G$ which is perfect and so that the stable commutator length
vanishes on $G=[G, G]$, for example $SL(n,\Z)$ for $n \geq 3$, it is
immediate that any {homogeneous semi-}length function is bounded,
and hence equivalent to the trivial { semi-}length function
$\normsymb(g)\equiv 0$.

More generally, for groups $G$ for which the stable commutator length
vanishes on $[G,G]$, we can deduce an analogue of Theorem~\ref{class}.
Note that there are several interesting examples of such groups,
including solvable groups, and more generally, amenable groups.

\begin{theorem}\label{quasi}
Let $G$ be a group such that the stable commutator length vanishes on
$[G, G]$ and assume $\normsymb: G \to \R$ satisfies \eqref{ng} and \eqref{double}.
Then there exist a real Banach space $\mathbb{B} = (\mathbb{B},\| \, \|)$
and a group homomorphism $\phi : G \to \mathbb{B}$ such that $\normsymb$
is equivalent to $x\mapsto \| \phi(x) \|$.
\end{theorem}
\begin{remark}
  {As in Remark~\ref{quasilength}, we can replace \eqref{ng} by the \emph{a priori}
  weaker condition that $\norm{xy} \leq \norm{x} + \norm{y} + k$ for all $x, y\in G$ with $k$ fixed.}
\end{remark}
\begin{proof}

Let $\ab: G \to G_{\ab}=  G / [G, G]$ be the abelianization homomorphism.
We first construct a homogeneous { semi-}length function $\bar{\normsymb}$
on { $G_{\ab}$} so that $\normsymb$ is equivalent to $\bar{\normsymb}\circ\ab$.
Let $\eta : G_{\ab} \to G$ be a section of $\ab$ { and let $\bar{\normsymb}_0(x) := \norm{\eta(x)} + c$}. We show that
$\bar{\normsymb}_0$ is a { { semi-}length} function. The required
$\bar{\normsymb}$ will be obtained by homogenizing $\bar{\normsymb}_0$.

By Proposition~\ref{commutatorbound}, as the stable commutator length
vanishes on $[G, G]$, it follows that  for $x, y\in G$, if $\ab(x)
= \ab(y)$, then $|\norm{x} - \norm{y}| \leq c$. Now, for
$\alpha, \beta \in G_{\ab}$, $\ab(\eta(\alpha\beta)) =
\ab(\eta(\alpha)\eta(\beta))$, hence
\[
|\norm{\eta(\alpha\beta)} - \norm{\eta(\alpha)\eta(\beta))}| \leq c.
\]
This together with { the triangle inequality~\eqref{ng} gives
\[
	\bar{\normsymb}_0(\alpha\beta) \leq \bar{\normsymb}_0(\alpha) + \bar{\normsymb}_0(\beta) + c,
\]
while using \eqref{double} instead} gives the required lower bound for $\bar{\normsymb}_0(\alpha^2)$.

Next, for $x\in G$, as $\ab(\eta(\ab(x))) = \ab(x)$, we have
{ $|\normsymb(x) - (\bar{\normsymb}_0 \circ \ab)(x)| \leq c$}. Thus
$\normsymb$ is equivalent to $\bar{\normsymb}_0 \circ \ab$.

Since $(\alpha \beta)^n = \alpha^n \beta^n$ in $G_{\ab}$, we also have
$\bar{\normsymb}_0((\alpha\beta)^n) \leq \bar{\normsymb}_0(\alpha^n) +
\bar{\normsymb}_0(\beta^n) + { c}$. We deduce that the homogenization
$\bar{\normsymb}$ of $\bar{\normsymb}_0$ is a { semi-}length
function on $G_{\ab}$, which is equivalent to $\bar{\normsymb}_0$ { due to the bounds~\eqref{ng-iter} and~\eqref{rearrange}, applied to $\bar{\normsymb}_0$}. Therefore
also $\normsymb$ is equivalent to $\bar{\normsymb}\circ \ab$ on $G$.

The claim now follows upon applying Theorem~\ref{class} to
$(G_{\ab},\bar{\normsymb})$ and taking $\phi$ to be the composition $G
\to G_{\ab} \to \mathbb{B}$.
\end{proof}

The following examples of length functions on the free group show
that some hypotheses are needed to get bounds as strong as those of the
Theorem (naturally the stable commutator length does not vanish in the
free group).
For example, consider the word $[a^k, b^m]$ in the free group $\free{2}$,
generated by $a$ and $b$, for some integers $k$ and $m$.
\begin{itemize}
\item The norm of such an element with respect  to the word metric is
$2(|k| + |m|)$.

\item If we have a length function $\normsymb$ which is symmetric and
conjugation-invariant, but not necessarily homogeneous, then we have the
bound $\norm{[a^k, b^m]} \leq 2 \min (|k|\, \norm{a}, |m|\, \norm{b})$.
Furthermore, { the $\normsymb_{wc;a,b}$ from above} are conjugation-invariant length
functions for which these inequalities hold with equality.

Further, $\norm{[a^k, b^m]} \geq 2 \min (|k|\, \norm{a}, |m|\, \norm{b})$
as, for any matching $M$ for $w=[a^k, b^m]$, if some pair $(i, j)$
corresponds to letters $a$ and $a^{-1}$, then no pair corresponds to
letters $b$ and $b^{-1}$ and conversely. Further, it is easy to find a
matching for $w$ for which the deficiency is $\min (|k|\, \norm{a}, |m|\,
\norm{b})$.  On the other hand, $\normsymb_{wc;a,b}$ is not homogeneous; for
instance, $\norm{[a,b]} = 2$ and $\norm{[a,b]^3} = 4$.  Similarly, we
have $\norm{[a^k,b^k]}=2|k|$ and $\norm{[a^k,b^k]^3} \leq 4|k|$, which
demonstrates that $2\norm{x} - \norm{x^2}$ is unbounded (as must be the
case, according to \eqref{9bound}).

\item On the other hand, the function $\normsymb_{cyc}$ associating to each
word the length of its cyclically reduced form is homogeneous, but not a
{ { semi-}length function}. For this we have $\normsymb_{cyc}([a^k, b^m]) = 2(|k| +
|m|)$.
\end{itemize}

Observe that all of the bounds on $\norm{[a^k, b^m]}$ here become
unbounded as $k,m \to \infty$.  This should be compared with
{ Proposition~\ref{main}, which establishes a bound
$\norm{[a^k,b^m]} \leq 5c$} that is uniform in $k$ and $m$ for any
function $\normsymb$ satisfying the hypotheses of that proposition.

\subsection{Finite balls in free groups}

From Proposition~\ref{main} and a standard compactness argument, we can
establish the following local version of the theorem.

\begin{theorem}
For any $\eps>0$ there exists $R \geq 4$ with the following property: if
$a,b$ are two elements of a group $G$, $B_{a,b}(R) \subset G$ is the
collection of all words in $a,b,a^{-1},b^{-1}$ of length at most $R$
\textup{$($}so in particular $B_{a,b}(R)$ contains $[a,b]$\textup{$)$},
and the map $\normsymb: B_{a,b}(R) \to [0,+\infty)$ is a ``local
{ semi}-length function'' which obeys the triangle inequality
\begin{equation}\label{tri}
 \norm{ xy } \leq \norm{x} + \norm{y}
\end{equation}
whenever $x,y,xy \in B_{a,b}(R)$, with equality when $x=y$, then one has
\[
\norm{ [a,b] } \leq \eps ( \norm{a} + \norm{b} ).
\]
\end{theorem}

\begin{proof}
By pulling back to the free group $\free{2}$ generated by $a$ and $b$, we
may assume without loss of generality that $G = \free{2}$.  Without loss
of generality we may also normalize $\norm{a}+\norm{b}=1$.  If the claim
failed, then one could find a sequence $R_n \to \infty$ and local
pseudo-length functions $\normsymb_n: B_{a,b}(R_n) \to [0,+\infty)$ such
that $\normn{a}+\normn{b} = 1$, but that $\normn{[a,b]} \geq \eps$.  By
the Arzela--Ascoli theorem, we can pass to a subsequence that converges
pointwise to a homogeneous pseudo-length function $\normsymb: G \to
[0,+\infty)$ such that $\norm{[a,b]} \geq \eps$, which contradicts
Proposition~\ref{main}.
\end{proof}

\begin{remark}
By carefully refining the arguments in the previous section, choosing $n$
to be various small powers of $R$ instead of sending $n$ to infinity, one
can extract an explicit value of $R$ of the form $R = C \eps^{-A}$ for
some absolute constants $C, A > 0$; we leave the details to the
interested reader.
\end{remark}

On the other hand, for any finite $R$ one can construct local length
functions $\normsymb: B(0,R) \to [0,+\infty)$ such that $\norm{x} > 0$
for all $x \in B(0,R) \setminus \{e\}$.  One construction is as follows.
Any two matrices $U_a, U_b\in SO(3)$ define a representation $x \mapsto
U_x$ of the free group $\free{2}$ in the obvious fashion.  Every $U_x$ is
then a rotation around some axis in $\R^3$ by some angle $0 \leq \theta_x
\leq \pi$ in one of the two directions around that axis; if $U_a$ and
$U_b$ are sufficiently close to the identity, then the angle $\theta_x$
is at most $\pi/2$ for all $x \in B(0,R)$. We set $\norm{x} \coloneqq
\theta_x$ for $x \in B(0,R)$.  Also, if $U_a,U_b$ are chosen generically,
the representation is faithful, as follows from the dominance of word
maps on simple Lie groups such as $SO(3)$, see \cite{borel}. Hence
$\norm{x}>0$ for any non-identity $x$.  From the triangle inequality for
angles we thus have \eqref{tri} whenever $x,y,xy \in B(0,R)$, with
equality when $x=y$.  Note that as one sends $R \to \infty$, the local
length functions constructed here converge to zero pointwise, so in the
limit we do not get any counterexample to the main theorem.



\end{document}